\documentclass[12pt]{article}
\usepackage[utf8]{inputenc}
\usepackage[T1]{fontenc}
\usepackage{lmodern}
\usepackage{graphicx}
\newcommand{\vs}{\vskip10pt}

\usepackage{amsmath}
\usepackage{amsthm}
\usepackage{amssymb}
\usepackage[export]{adjustbox}
\usepackage[margin=0.5in]{geometry}
\usepackage{setspace}
\usepackage{enumitem}
\usepackage{float}
\usepackage{tabularx,ragged2e}
\newtheorem{prop}{Proposition}

\newtheorem{thm}[prop]{Theorem}

\newtheorem{rem}[prop]{Remark}

\newtheorem*{theorem*}{Theorem}
\newtheorem*{lemma*}{Lemma}
\newtheorem*{prop*}{Proposition}

\usepackage{hyperref}
\hypersetup{
	colorlinks,
	citecolor=black,
	filecolor=black,
	linkcolor=black,
	urlcolor=black
}

\title{An Algorithm to Solve the Generalized Conjugacy Problem in Relatively Hyperbolic Groups}
\author{Christopher Karpinski \footnote{McGill University, email: christopher.karpinski@mail.mcgill.ca.}}

\newcommand\blfootnote[1]{%
	\begingroup
	\renewcommand\thefootnote{}\footnote{#1}%
	\addtocounter{footnote}{-1}%
	\endgroup
}

\begin{document}
	
	\maketitle
	
	\begin{abstract}
		
		In this note, we provide a (super-exponential time) algorithm to solve the generalized conjugacy problem in relatively hyperbolic groups, given solvability of the generalized conjugacy problem in each of the parabolic subgroups. 
		
	\end{abstract}
	
	
	\blfootnote{This work was supported by NSERC grant RGPIN-2016-06154.}
	
	\newpage
	
	\section{Introduction}
	
	\vs
	
	The \textit{generalized conjugacy problem (GCP)}, also known as the \textit{simultaneous conjugacy problem}, the \textit{conjugacy problem for finite lists}, the \textit{Whitehead problem for inner automorphisms} and the \textit{multiple conjugacy problem} asks whether there exists an algorithm such that given as input a recursive presentation for a recursively presented group $G$ and two finite lists $\{a_1,...,a_m\}, \{b_1,...,b_m\}$ of elements of $G$ decides whether or not these two lists are conjugate, that is, if there exists an element $x \in G$ such that $a_i^x = b_i$ for all $i = 1,...,m$ (recall that the notation $a^x$ means $x^{-1} a x$). The \textit{generalized conjugacy search problem (GCSP)} asks whether there exists an algorithm that can produce a conjugating element between two finite lists $\{a_1,...,a_m\}, \{b_1,...,b_m\}$ or determine that no such element exists.
	
	\vs 
	
	The GCP is known to be solvable in several classes of groups. For instance, the following classes of groups all admit solutions to the GCP: 
	
	\begin{itemize}
		\item Finitely generated virtually free groups (Ladra and Silva in [13]).
		\item Artin groups of large type (Bezverkhnii in [4]).
		\item Braid groups and general linear groups (Hahn, Lee and Park in [11]).
		\item Generalized Baumslag-Solitar groups (Beeker in [3]).
		\item Polycyclic groups, where the solution to the GCP reduces to repeated application of the solution to the conjugacy problem for single elements (Eick and Kahrobaei in [9]).
	\end{itemize}
	
	\vs
	
	Besides the above, an important class of groups for which the GCP is solvable is the class of hyperbolic groups, as was first shown by Bridson and Howie in [6]. While Bridson and Howie solved the GCP in hyperbolic groups, they were unable to produce a sub-exponential time algorithm that dealt with lists of torsion elements. In [7], using results from [10], a linear run-time algorithm was constructed by Buckley and Holt that solves the GCP for lists of torsion elements as well, thereby establishing a linear run-time algorithm to solve the GCP in hyperbolic groups. 
	
	\vs 
	
	In this paper, we consider the solvability of the GCP in the more general class of \textit{relatively hyperbolic groups} (which we review below). To the author's knowledge, a solution to the GCP has never been recorded anywhere. There have, however, been solutions to closely related problems in relatively hyperbolic groups (or certain subclasses of relatively hyperbolic groups). For instance, in [12], Kharlampovich and Ventura show that the Whitehead problem (which asks if there exists an algorithm such that given finite lists $\{a_1,...,a_m\}, \{b_1,...,b_m\}$ of elements of $G$ determines if there exists an automorphism of $G$ taking each $a_i$ to $b_i$) is solvable in toral relatively hyperbolic groups (i.e. torsion-free groups hyperbolic relative to abelian subgroups). There have also been several papers on the solvability of the (single) conjugacy problem (i.e. the conjugacy problem for single elements) in relatively hyperbolic groups. In [8], Bumagin shows that the conjugacy problem in relatively hyperbolic groups can be solved in polynomial time provided the conjugacy problem is solvable in polynomial time in each of the parabolic subgroups. The solvability of the conjugacy problem in relatively hyperbolic groups was also discussed in an earlier paper of O'Connor in [14], as well as in the paper [2] by Antolin and Sale. In the paper [2], the  conjugacy problem is solved using a variation of the conjugacy length function called the \textit{permutation conjugacy length function}. Antolin and Ciobanu also gave an algorithm in [1] to solve the conjugacy problem in groups hyperbolic relative to abelian subgroups using bounded conjugacy diagrams, which generalizes an algorithm given in [6] to solve the conjugacy problem in hyperbolic groups using bounded conjugacy diagrams. However, neither of these approaches to solve the conjugacy problem generalize to give a solution to the GCP, as we discuss at the end of this paper.  The goal of this paper is to extend the work of [1], [2], [8] and [14] to give an algorithm that solves the GCP in relatively hyperbolic groups, provided the parabolic subgroups have solvable GCP. 
	
	\vs 
	
	\section{Background}
	
	Let us define relatively hyperbolic groups and review some basic terminology. 
	
	\vs 
	
	Let $G$ be a group and $\{H_{\lambda}\}_{\lambda \in \Lambda}$ a collection of subgroups of $G$ and suppose $X$ is a set such that $G$ is generated by $X \cup \mathcal{H}$, where $\mathcal{H} = \cup_{\lambda \in \Lambda} H_{\lambda}$. In the Cayley graph $\Gamma(G; X \cup \mathcal{H})$, edges may be labeled by an element from $X$ or an element from a subgroup $H_{\lambda}$. Given a path $\alpha$ in $\Gamma(G; X \cup \mathcal{H})$ (i.e. a sequence of edges in $\Gamma(G; X \cup \mathcal{H})$),  the \textbf{label} of $\alpha$ is the word in $X \cup \mathcal{H}$ obtained as the concatenation of the edge labels of $\alpha$ reading from $\alpha_-$ to $\alpha_+$. A subpath $p$ of $\alpha$ whose label is a word in some $H_{\lambda}$ is called an $\mathbf{H_{\lambda}}$\textbf{-subpath} of $\alpha$. A maximal $H_{\lambda}$-subpath of $\alpha$ is called an $\mathbf{H_{\lambda}}$\textbf{-component} of $\alpha$. By a \textbf{(parabolic) component} of $\alpha$, we mean an $H_{\lambda}$-component of $\alpha$ for some $\lambda \in \Lambda$. Similarly, given a word $w$ in the alphabet $X \cup \mathcal{H}$, a subword $v$ of $w$ is called an $\mathbf{H_{\lambda}}$\textbf{-subword} of $w$ if is a word in the alphabet $H_{\lambda}$. An $\mathbf{H_{\lambda}}$\textbf{-syllable} of $w$ is a maximal $H_{\lambda}$ subword of $w$. A \textbf{syllable} of $w$ is an $H_{\lambda}$-syllable of $w$ for some $\lambda \in \Lambda$. 
	
	\vs
	
	Given a path $\alpha$, we use the notation $\alpha_-$ and $\alpha_+$ to denote the origin (i.e. start) and terminus (i.e. end), respectively, of $\alpha$. Two $H_{\lambda}$-components $p,q$ of a path $\alpha$ in $\Gamma(G; X \cup \mathcal{H})$ are called \textbf{connected} if there exist $h_1, h_2 \in H_{\lambda}$ such that $p_- h_1 = q_-$ and $p_+ h_2 = q_+$. If an $H_{\lambda}$-component $p$ of $\alpha$ is not connected to any other $H_{\lambda}$-component of $\alpha$, then $p$ is called an \textbf{isolated} $\mathbf{H_{\lambda}}$\textbf{-component} of $\alpha$. A path $\alpha$ that contains two connected $H_{\lambda}$-components for some $\lambda \in \Lambda$ is said to \textbf{backtrack}, otherwise it is said to be a path \textbf{without backtracking}.
	
	\vs 
	
	We let $d$ denote the natural path metric on the Cayley graph $\Gamma(G; X \cup \mathcal{H})$. Given a path $\alpha$, we let $\ell(\alpha)$ denote the length of the path (i.e. the number of edges in $\alpha$). For $C \geq 1$ and $D \geq 0$, a path $\alpha$ in $\Gamma(G; X \cup \mathcal{H})$ is called a $\mathbf{(C,D)}$\textbf{-quasi-geodesic} if for any subpath $p$ of $\alpha$ we have $\ell(p) \leq Cd(p_-, p_+) + D$. Note that a geodesic path in $\Gamma(G; X \cup \mathcal{H})$ is a (1,0)-quasi-geodesic. 
	
	\vs 
	
	We say that $G$ is \textbf{weakly hyperbolic relative to} the collection $\{H_{\lambda}\}_{\lambda \in \Lambda}$ if there exists a finite generating set $X$ of $G$ such that $\Gamma(G; X \cup \mathcal{H})$  is a hyperbolic metric space (i.e. satisfies the slim triangle condition, see page 399 of [5]). We refer to the word metric on $G$ from the generating set $X$ as $d_X$ (as opposed to the metric from the generating set $X \cup \mathcal{H}$ which we denote by $d$). If $G$ is weakly hyperbolic relative to $\{H_{\lambda}\}_{\lambda \in \Lambda}$, then we call the subgroups $\{H_{\lambda}\}_{\lambda \in \Lambda}$ the \textbf{parabolic subgroups}. 
	
	\vs 
	
	We say that $G$ is \textbf{hyperbolic relative to } $\{H_{\lambda}\}_{\lambda \in \Lambda}$ if $G$ is weakly hyperbolic relative to $\{H_{\lambda}\}_{\lambda \in \Lambda}$ and $\Gamma(G; X \cup \mathcal{H})$ satisfies the following property. For any $C \geq 1$ and $D \geq 0$ there exists a constant $k$ depending only on $C,D$ such that for any $(C,D)$-quasi-geodesics $\alpha, \beta$ without backtracking and with the same endpoints (i.e. $\alpha_- = \beta_-$ and $\alpha_+ = \beta_+$), if $p$ is a component of $\alpha$ or $\beta$ which is not connected to any component of $\alpha$ or $\beta$  then $d_X(p_-, p_+) \leq k$. This property is known as \textbf{bounded coset penetration property }(BCP for short). (This is known as Farb's definition of relatively hyperbolic groups. In his paper [15], Osin uses a different definition using (relative) Dehn functions that applies for non-finitely generated groups as well, but Farb's and Osin's definitions are equivalent for finitely generated groups). 
	
	\vs
	
	Throughout this paper, we fix a group $G$ generated by a finite set $X$ hyperbolic relative to finitely presented subgroups $\{H_1,...,H_N\}$. We assume that the GCP is solvable in each parabolic subgroup $H_i$. We assume that $X$ satisfies the conditions of the following lemma from [15]: 
	
	\begin{lemma*} [15, Lemma 3.1]
		There exists a constant $A$ such that the following holds. Let G be a finitely generated group hyperbolic relative to a collection of subgroups $\{H_1,...,H_m\}$. Then there exists a finite generating set $X$ of G satisfying the following condition. Let $q$ be a cycle in $\Gamma(G;X \cup \mathcal{H})$, $ p_1, . . . , p_k$ a set of isolated $H_i$-components of $q$. Then $\sum_{i=1}^k d_X((p_i)_-, (p_i)_+) \leq A \ell(q)$. 
	\end{lemma*}

	 (As shown in [15], such finite generating set $X$ always exists). We also assume  that $X$ is symmetrized, that is, $X = X^{-1}$. We will denote $\Gamma := \Gamma(G;X)$ the Cayley graph of $G$ with respect to $X$ and $\hat{\Gamma} := \Gamma(G; X \cup \mathcal{H})$ the Cayley graph of $G$ with respect to $X \cup \mathcal{H}$. Given an element $g \in G$, as usual $\vert g \vert_X, \vert g \vert_{X \cup \mathcal{H}}$ denote the word lengths of $g$ with respect to the generating sets $X$ and $X \cup \mathcal{H}$, respectively. The notation $B_r(x)$ denotes the closed ball of radius $r$ about the point $x$ in $\hat{\Gamma}$ with its natural combinatorial metric and similarly $B_r^X(x)$ denotes the closed ball of radius $r$ about the point $x$ in $\Gamma$ with its natural combinatorial metric. 
	
	\vs 
	
	Let $\alpha, \beta$ be a pair of $(\lambda, \varepsilon)$-geodesics in $\hat{\Gamma}$ for $\lambda \geq 1$ and $\varepsilon \geq 0$. We call $\alpha, \beta$ $\mathbf{k}$\textbf{-similar} (for $k \geq 0$)  if $\max\{d_X(\alpha_-, \beta_-), d_X(\alpha_+, \beta_+)\} \leq k$. We call $\alpha, \beta$ \textbf{symmetric} if they have the same label. Given vertices $u$ on $\alpha$ and $v$ on $\beta$, we say that $u,v$ are \textbf{synchronous vertices} if $\ell([\alpha_-, u]) = \ell([\beta_-, v])$ (where $[\alpha_-, u]$ and $[\beta_-, v]$ denote the segments along $\alpha, \beta$, respectively between $\alpha_-, u$ and $\beta_-, v$, respectively).  Given two  $H_i$-components $p,q$ of $\alpha, \beta$, respectively, we say that $p,q$ are \textbf{synchronous components} if the vertices $p_-, q_-$ are synchronous. 
	
	
	\vs 
	
	\section{Main results}
	
	\vs
	
	Our goal will be to prove the following theorem.
	
	\begin{thm}
		
		There exists a recursive function $f: \mathbb{N} \rightarrow \mathbb{N}$ such that the following holds. Let \newline$\{a_1,...,a_m\}, \{b_1,...,b_m\}$ be conjugate lists of elements in $G$ and let $\mu = \max \{\vert a_i \vert_X, \vert b_i \vert_X: i = 1,...,m\}$. Then there exists an element $x$ conjugating $\{a_1,...,a_m\}$ to $\{b_1,...,b_m\}$ such that $\vert x \vert_X \leq f(\mu)$. 
		
	\end{thm}
	
	Theorem 1 immediately produces an algorithm to solve the GCP in $G$ (in fact, the GCSP), namely, for each $x \in B^X_{f(\mu)}(1)$ check if $a_i^x = b_i$ for all $i = 1,...,m$ using the solution to the word problem in $G$ (which exists since the word problem is solvable in each $H_i$; see section 5.1 in [15]). If these checks fail for all $x \in B^X_{f(\mu)}(1)$ then by Theorem 1 the lists are not conjugate, otherwise the lists are conjugate. The function $f$ that we will produce will be super-exponential in $\mu$. We shall not attempt to produce algorithms with optimal efficiency in this paper. 
	
	\vs
	
	To prove Theorem 1, we will make use of the following results from Section 3.4 of [15]:
	\vs 
	
	\begin{theorem*} [15, Theorem 3.37]
		For any $k \geq 0$, there exists a constant $\chi = \chi(k)$ with the following property. Let $(p,q)$ be a symmetric pair of $k$–similar geodesics in $\hat{\Gamma}$, $u, v$ a pair of synchronous vertices on $p$ and $q$, respectively. Then $d_X(u, v) \leq \chi$.
	\end{theorem*}

	\begin{lemma*} [15, Lemma 3.43]
		For any $k \geq 0, \lambda \geq 1, c \geq 0$, there exists a constant $\eta = \eta(\lambda,c,k)$ such that the following condition holds. Let $(p, q)$ be an arbitrary symmetric pair of $k$–similar $(\lambda, c)$–quasi–geodesics without backtracking in $\hat{\Gamma}$ such that no synchronous components of $p$ and $q$ are connected. Then for any $i = 1, . . . , N$ and any $H_i$–component $e$ of $p$, we have $d_X (e_-, e_+) \leq \eta \ell([p_-, e_+])$ where $[p_-,e_+]$ is the segment of $p$ between $p_-$ and $e_+$.
	\end{lemma*}
	
	\vskip30pt
	
	{\huge \textbf{Acknowledgements}}
	
	\vs
	
	The author is very grateful to Inna Bumagin for introducing the author to the subject of this note, for continuous support and advice, for providing helpful suggestions on earlier drafts of this note and for many useful discussions which have aided the author tremendously in producing this work. 
	
	\vskip30pt
	
	\section{Proof of Theorem 1}
	
	\subsection{Bounding the relative length of conjugating elements}
	
	We first show that for arbitrary conjugate lists $\{a_1,...,a_m\}, \{b_1,...,b_m\}$ we may bound above the relative length of a shortest (in relative length) conjugating element. 
	
	\vs 
	
	\begin{thm}
		
		Let $\{a_1,...,a_m\}, \{b_1,...,b_m\}$ be two conjugate lists. Let $\mu = \max \{\vert a_i \vert_X, \vert b_i \vert_X: i = 1,...,m\}$. Let $x$ be an element of minimal relative length such that $\{a_1,...,a_m\}^x = \{b_1,...,b_m\}$. Then $\vert x \vert_{X \cup \mathcal{H}} \leq (\vert X \vert^{\chi(\mu)} + 1)^{\vert X \vert^{\mu} + 1}$, where $\chi(\mu)$ is the constant from Theorem 3.37 in [15] (stated above). 
		
	\end{thm}

	\begin{proof}
		
		We follow a similar idea to the proof of Theorem 3.3 in [6].
		
		\vs
		
		First, note that we may assume that the $a_i$ are all distinct, as if there are any repetitions in the list, then we may simply remove the repeated elements from $\{a_1,...,a_m\}$ and $\{b_1,...,b_m\}$ and this will not affect the conjugacy of the original list (the list with repeated elements). 
		
		\vs 
		
		We begin by fixing a geodesic word representing $x$: $x = g_1 g_2 ... g_{\ell}$, where $g_i \in X \cup \mathcal{H}$ for all $i=1,...,\ell$. We shall also denote $x_j := g_1...g_j$ for each $j = 1,...,\ell$. Suppose, for contradiction, that $\ell > (\vert X \vert^{\chi(\mu)} + 1)^{\vert X \vert^{\mu} + 1}$. 
		
		\vs 
		
		Note that for each $i = 1,...,m$, the equality $a_i^x = b_i$ produces a geodesic quadrilateral $Q_i$ in the relative Cayley graph $\hat{\Gamma}$ as shown in Figure 1, where the $a_i, b_i$ are represented by relative geodesics, which are labeled by fixed geodesic words over $X \cup \mathcal{H}$ that we have chosen for the $a_i, b_i$. In each $Q_i$, the yellow paths $\alpha_i, \beta_i$ (which form the top and bottom edges, respectively, of the geodesic quadrilateral $Q_i$) are $\mu$-similar symmetric geodesics in $\hat{\Gamma}$ labeled by the geodesic word $g_1...g_{\ell}$ representing $x$. 
		
		\vs 
		
		For any $j = 1,...,\ell$, we examine the synchronous subpaths $p_{ij} \subseteq \alpha_i$ and $q_{ij} \subseteq \beta_i$ which are each labeled by the subword $x_j$ of $x$ (see Figure 1). Denoting $u_{ij} : = (p_{ij})_+, v_{ij} := (q_{ij})_+$, we have that $u_{ij}, v_{ij}$ are synchronous vertices on $\alpha_i, \beta_i$ so by Theorem 3.37 in [15] we have $d_X(u_{ij},v_{ij}) \leq \chi(\mu)$, where $\chi(\mu)$ is a constant from Theorem 3.37 in [15] depending only on $\mu$. Note that $v_{ij}^{-1}u_{ij} = a_i^{x_j}$, so we obtain $\vert a_i^{x_j} \vert_X \leq \chi(\mu)$ for each $i = 1,...,m$ and $j = 1,...,\ell$.

\begin{figure} [H]
	\centering
	\includegraphics[width=0.6\linewidth]{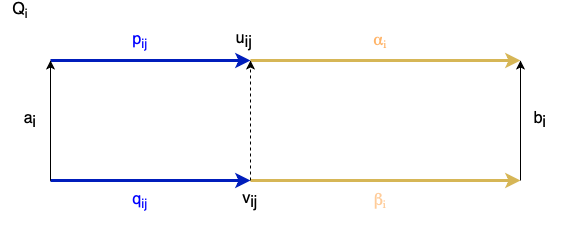}
	\caption{The geodesic quadrilateral $Q_i$ representing the conjugacy $a_i^x = b_i$ and the synchronous subpaths $p_{ij}, q_{ij}$ representing the subword $x_j$ of $x$}
	\label{fig:figure-1-2}
\end{figure}

		
		\vs 
		
		Thus, for each $j = 1,...,\ell$, the tuple $(a_1^{x_j},...,a_m^{x_j})$ is an element of $(B^X_{\chi(\mu)}(1))^m$, the cardinality of which is bounded above by $(\vert X \vert^{\chi(\mu)} + 1)^m$. Let us now bound above $m$ in terms of $\mu$. 
		
		\vs 
		
		 As $\vert b_i \vert_X \leq \mu$ for all $i$ and since the $b_i$ are all distinct (as the $a_i$ are all distinct and the lists are conjugate), we have $m \leq \vert B^X_{\mu}(1) \vert \leq \vert X \vert^{\mu} + 1$. 
		 
		 \vs 
		 
		 Therefore, as $j$ ranges through $\{1,...,\ell\}$, the number of distinct tuples $(a_1^{x_j},...,a_m^{x_j})$ that may be produced is at most $(\vert X \vert^{\chi(\mu)} + 1)^{\vert X \vert^{\mu} + 1}$. Since $\ell > (\vert X \vert^{\chi(\mu)} + 1)^{\vert X \vert^{\mu} + 1}$, by the Pigeonhole Principle, there exist $1 \leq s < t \leq \ell$ such that $(a_1^{x_s},...,a_m^{x_s}) = (a_1^{x_t},...,a_m^{x_t})$, i.e. $a_i^{x_s x_t^{-1}} = a_i$ for all $i = 1,...,m$. This yields that $x_s x_t^{-1} x$ conjugates $\{a_1,...,a_m\}$ to $\{b_1,...,b_m\}$. However, note that $\vert x_s x_t^{-1} x \vert_{X \cup \mathcal{H}} \leq \vert x_s \vert_{X \cup \mathcal{H}} + \vert x_t^{-1} x \vert_{X \cup \mathcal{H}} = s + \ell - t < \ell$, contradicting the minimality of $\ell$. 
		 
		 \vs 
		 
		 We therefore conclude that $\ell \leq (\vert X \vert^{\chi(\mu)} + 1)^{\vert X \vert^{\mu} + 1}$.
		
	\end{proof}

	\begin{rem}
		
		Note that Theorem 2 gives a short proof of the solvability of the GCP in $G$ when $G$ is a hyperbolic group (where $\mathcal{H} = \{1\}$). 
		
	\end{rem}
	
	\subsection{Bounding parabolic components of conjugating elements}
	
	We bound parabolic components of minimal relative length conjugating elements to establish the following theorem, which yields Theorem 1.
	
	\begin{thm}
		
		There is a recursive function $g: \mathbb{N} \rightarrow \mathbb{N}$ such that the following holds. Suppose that $\{a_1,...,a_m\}, \{b_1,...,b_m\}$ are conjugate lists of elements. Let $\mu = \max \{\vert a_i \vert_X, \vert b_i \vert_X : i = 1,...,m\}$. There exists a conjugating element $x$ between the two lists such that $\vert x \vert_X \leq (\vert X \vert^{\chi(\mu)} + 1)^{\vert X \vert^{\mu} + 1} g(\mu)$. 
		
	\end{thm}
	
	\begin{proof}
		
		Let $x$ be a shortest (in relative length) conjugating element between the two lists. By Theorem 2, we have $\vert x \vert_{X \cup \mathcal{H}} \leq (\vert X \vert^{\chi(\mu)} + 1)^{\vert X \vert^{\mu} + 1}$. Fixing a geodesic word $w$ representing $x$, as in the proof of Theorem 2, we have $m$ geodesic quadrilaterals representing simultaneous conjugacy diagrams as shown in Figure 2, where each orange geodesic path has label the fixed geodesic word that we have chosen for $x$ and each $a_i, b_i$ are represented by fixed geodesics. In each geodesic quadrilateral the geodesics whose label represents $x$ are a symmetric pair of $\mu$-similar geodesics.

		
		\vs 
		
		We consider two cases:
		
		\vs 
		
		\underline{Case 1:} There does not exist an $H_i$-syllable of $w$ such that the $H_i$-component corresponding to this syllable is connected to its synchronous component in each quadrilateral $Q_j$. Then for each syllable $w$, we may choose some quadrilateral $Q_i$ such that the component labeled $h$ corresponding to this syllable is not connected to its synchronous component on the opposite side of $Q_i$. Looking at the symmetric pair of geodesics labeled by this component $h$, these geodesics are $\chi(\mu)$-similar by Theorem 3.37 in [15] because their endpoints are synchronous vertices on the large geodesic labeled $x$. By Lemma 3.43 in [15] (stated above) we have $\vert h \vert_X \leq \eta(1,0,\chi(\mu)) \vert h \vert_{X \cup \mathcal{H}} = \eta(1,0,\chi(\mu)$. Thus, we may bound above the $X$-length of each syllable of $w$ by $\eta(1,0,\chi(\mu))$. We therefore have that $\vert x \vert_X \leq \vert x \vert_{X \cup \mathcal{H}} \eta(1,0,\chi(\mu)) \leq (\vert X \vert^{\chi(\mu)} + 1)^{\vert X \vert^{\mu} + 1} \eta(1,0,\chi(\mu))$. 
		
		\vs 
		
		\underline{Case 2:} There exists a syllable of $w$ such that the component corresponding to this syllable is connected to its synchronous component in each quadrilateral $Q_i$. Let $h_1,...,h_n$ be the labels all of the components which are connected to their synchronous components in each of the quadrilaterals $Q_i$ (see Figure 2) (note that the $h_i$ may belong to different parabolic subgroups). 
		
		\vs

\begin{figure} [H]
	\centering
	\includegraphics[width=0.6\linewidth]{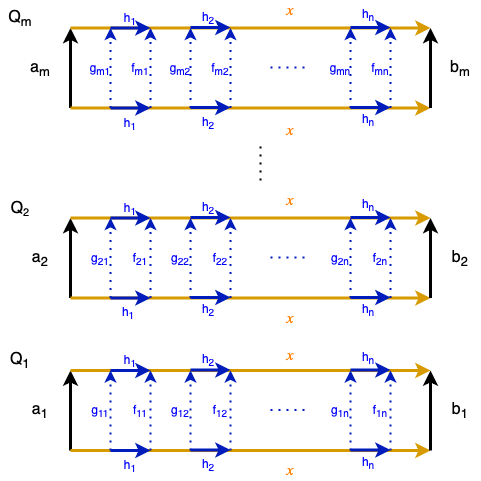}
	\caption{The case of connected synchronous components in each conjugacy diagram $Q_i$ (Case 2)}
	\label{fig:figure-3-2}
\end{figure}

		
		\vs 
		
		Denote the connectors between the component labeled $h_j$ with its synchronous component in $Q_i$ by $g_{ij}$ and $f_{ij}$ for each $i = 1,...,m$ and $j = 1,...,n$ (see Figure 2).
		
		\vs 
		
		We have that for each $j = 1,...,n$, $\{g_{1j}, g_{2j},..., g_{mj}\}^{h_j} = \{f_{1j}, f_{2j},..., f_{mj}\}$. Therefore, the lists \newline$\{g_{1j}, g_{2j},..., g_{mj}\}, \{f_{1j}, f_{2j},..., f_{mj}\}$ are elements in some parabolic subgroup $H$ which are conjugate by an element of $H$. Since the GCP is solvable in $H$, there exists a conjugating element $s_j \in H$ such that $\{g_{1j}, g_{2j},..., g_{mj}\}^{s_j} = \{f_{1j}, f_{2j},..., f_{mj}\}$ and $\vert s_j \vert_X \leq \theta(\mu)$ for some recursive function $\theta$ (this $\theta$ depends on the parabolic subgroup but we may just take the max of the $\theta$s over each of the parabolic subgroups to remove this dependence on the subgroups, and also the argument of $\theta$ should depend on $\max \{\vert g_{ij} \vert_X, \vert f_{ij} \vert_X : i = 1,...,m\}$, but we have $\vert g_{ij} \vert_X, \vert f_{ij} \vert_X \leq \chi(\mu)$ for all $i,j$ so we may assume the upper bound is a function of $\mu$). Let $x' \in G$ be formed by replacing each label $h_j$ by $s_j$. Note that $\{a_1,...,a_m\}^{x'} = \{b_1,...,b_m\}$ (because $s_j$ conjugates each $\{g_{1j}, g_{2j},..., g_{mj}\}$ to $\{f_{1j}, f_{2j},..., f_{mj}\}$) and $\vert x' \vert_{X \cup \mathcal{H}} = \vert x \vert_{X \cup \mathcal{H}}$. In $x'$ we may bound the $X$-length of the $s_j$ by $\theta(\mu)$ and for the components of the geodesics representing $x'$ which are not connected to their synchronous component in some $Q_i$, then as in Case 1 we can bound the $X$-length of such components by $\eta(1,0, \chi(\mu))$. Thus, each component of our word representing $x'$ has $X$-length bounded above by $\max \{\theta(\mu), \eta(1,0, \chi(\mu))\}$. Hence $\vert x' \vert_X \leq \vert x' \vert_{X \cup \mathcal{H}} \max \{\theta(\mu), \eta(1,0, \chi(\mu))\} = \vert x \vert_{X \cup \mathcal{H}} \max \{\theta(\mu), \eta(1,0, \chi(\mu))\} \leq (\vert X \vert^{\chi(\mu)} + 1)^{\vert X \vert^{\mu} + 1} \max \{\theta(\mu), \eta(1,0, \chi(\mu))\}$. 
		
		\vs 
		
		Therefore, we set $g = \max \{\theta, \eta(1,0,\chi(-))\}$. Note that this is a recursive function since $\theta$ is a recursive function (by the solvability of the GCP in each of the parabolic subgroups), $\eta(1,0,-)$ is a recursive function by the proof of Lemma 3.43 in [15] and $\chi$ is a recursive function by the proof of Theorem 3.37 in [15]. 
		
	\end{proof}

	\begin{rem}
		Note that if the parabolic subgroups are abelian, then we do not have Case 2 in the proof of Theorem 4, as if we had Case 2, then we would have $g_{ij} = f_{ij}$ for all $i,j$ because they are conjugate elements in an abelian group. Therefore, we could obtain a new conjugating element between $\{a_1,...,a_m\}, \{b_1,...,b_m\}$ by deleting each of the connected synchronous components shown in Case 2. This would be a shorter conjugating element than $x$, a contradiction.
	\end{rem}

	\vs 
	
	\textbf{Remark 6}: As a final remark, we discuss the difficulty in deducing a solution to the GCP in relatively hyperbolic groups from solutions to the conjugacy problem given in [1], [2], [8] and [14]. 
	
	\vs 
	
	Given two finite lists $\{a_1,...,a_m\}, \{b_1,...,b_m\}$ of elements of $G$, we can use the solutions to the conjugacy problem in the above papers to determine conjugacy of each pair $(a_i, b_i)$ and use the solution to the conjugacy search problem to find a conjugating element between these pairs if one exists. But if we find that each pair $(a_i, b_i)$ are conjugate and produce a conjugating element $x_i$ for the pair $(a_i, b_i)$, there does not seem to be a way in general to put the results for each of the pairs together to determine whether the entire lists are conjugate.
	
	\vs
	
	If one could compute one of the centralizers $C(a_i)$ and if one applies the solution to the conjugacy search problem to the pair $(a_i, b_i)$ to produce a conjugating element $x_i$ if one exists, then the GCP could be solved if one could find a way to efficiently check through the (possibly infinite) set $C(a_i)x_i$ for an element that conjugates the lists $\{a_1,...,a_m\}, \{b_1,...,b_m\}$ (as any element that conjugates the lists would have to reside in $C(a_i)x_i$ because a conjugating element between the lists would conjugate $a_i$ to $b_i$ in particular and $C(a_i)x_i$ is precisely the set of elements that conjugate $a_i$ to $b_i$). However, to the best of the author's knowledge, in general finitely generated relatively hyperbolic groups, there does not exist any way to compute the centralizers $C(a_i)$ for general elements $a_i$ (though these centralizers can be computed in special cases such as if $a_i$ is a parabolic element in a toral relatively hyperbolic group, in which case $C(a_i)$ is just the unique conjugate parabolic subgroup containing $a_i$). Furthermore, even if the centralizers could be computed, the author knows of no algorithm to determine if there is an element in $x_i C(a_i)$ that conjugates $\{a_1,...,a_m\}$ to $\{b_1,...,b_m\}$. The author believes that the best bet for such algorithms would be suitable generalizations of algorithms for hyperbolic groups that compute centralizers of elements and test for conjugating elements between the lists. Such algorithms for hyperbolic groups are given in [6] and [7], however, the author has not yet found a way to generalize these algorithms to relatively hyperbolic groups. Therefore, the solvability of the GCP in relatively hyperbolic groups does not clearly follow from the work of [1], [2], [8] nor [14].

	\newpage
	\section{References}
	
	[1] Antolin, Y, Ciobanu, L. Finite Generating Sets of Relatively Hyperbolic Groups and Applications to Geodesic Languages. \textit{Transactions of the American Mathematical Society} 368.11 : 7965–8010 (2016).
	
	\vs
	
	[2] Antolin, Y, Sale, A. Permute and conjugate: the conjugacy problem in relatively hyperbolic groups. \textit{The bulletin of the London Mathematical Society}, Vol.48(4), p.657-675 (2016).
	
	\vs
	
	[3] Beeker, B. Multiple conjugacy problem in graphs of free abelian groups. \textit{Groups Geom. Dyn}. 9, no. 1 pp. 1–27 (2015).
	
	\vs
	
	[4] Bezverkhnii, V.N. Decision of the generalized conjugacy problem in Artin groups of large type. \textit{Fundam. Prikl. Mat.}, 1999, Volume 5, Issue 1, Pages 1–38 (1999).
	
	\vs
	
	[5] M.R. Bridson and A. Haefliger. Metric Spaces of Non-Positive Curvature. Springer-Verlag, Berlin (1999).
	
	\vs
	
	[6] M.R. Bridson, J. Howie. Conjugacy of finite subsets in hyperbolic groups. \textit{International Journal of Algebra and Computation}, 15(4), 725-756 (2005). 
	
	\vs
	
	[7] Buckley, D.J., Holt D.F. The Conjugacy Problem in Hyperbolic Groups for Finite Lists of Group Elements. \textit{International Journal of Algebra and Computation}, 23(5), 1127-1150 (2013). 
	
	 \vs
	 
	 [8] Bumagin, I. Time complexity of the conjugacy problem in relatively hyperbolic groups. \textit{International Journal of Algebra and Computation}, 25(5):689-723 (2015). 
	 
	 \vs 
	 
	 [9] B. Eick, D. Kahrobaei, Polycyclic groups: A new platform for cryptology?, math.GR/0411077, Preprint (2004).
	 
	 \vs 
	 
	 [10] D.B.A. Epstein and D.F. Holt. The linearity of the conjugacy problem in word-hyperbolic groups. \textit{International Journal of Algebra and Computation}, 16(2):287-305 (2006).
	 
	 \vs
	 
	 [11] Geun Hahn, S., Lee, E., Hong Park, J. Complexity of the generalized conjugacy problem. \textit{Discrete Applied Mathematics}, 130(1), 33-36 (2003). 
	 
	 \vs
	 
	 [12] O. Kharlampovich, E. Ventura, A Whitehead algorithm for Toral Relatively Hyperbolic Groups. \textit{International Journal of Algebra and Computation,} Vol. 22, No. 08 (2012).
	 
	 \vs
	 
	 [13] Ladra, M., Silva, P. V. The generalized conjugacy problem for virtually free groups. \textit{Forum Mathematicum}, 23(3), 447-482 (2011). 
	 
	 \vs
	 
	 [14] Z. O’Connor. Conjugacy search problem for relatively hyperbolic groups. math.GR/1211.5561, Preprint (2012).
	 
	 \vs
	
	[15] Osin, D.V. Relatively hyperbolic groups: intrinsic geometry, algebraic properties, and algorithmic problems. \textit{Mem. Amer. Math. Soc.}, 179(843):vi+100 (2006).


\end{document}